\documentclass[10pt]{article}
\usepackage{amsmath}
\usepackage{amssymb}
\usepackage{amsthm}
\usepackage{mathrsfs}

\begin{document}
\title{Weak type estimates for intrinsic square functions on the weighted Morrey spaces}
\author{Hua Wang \footnote{E-mail address: wanghua@pku.edu.cn.}\\
\footnotesize{Department of Mathematics, Zhejiang University, Hangzhou 310027, China}}
\date{}
\maketitle

\begin{abstract}
In this paper, we will obtain the weak type estimates of intrinsic square functions including the Lusin area integral, Littlewood-Paley $g$-function and $g^*_\lambda$-function on the weighted Morrey spaces $L^{1,\kappa}(w)$ for $0<\kappa<1$ and $w\in A_1$.
\end{abstract}

\section{Introduction and main results}

Let ${\mathbb R}^{n+1}_+=\mathbb R^n\times(0,\infty)$ and $\varphi_t(x)=t^{-n}\varphi(x/t)$. The classical square function (Lusin area integral) is a familiar object. If $u(x,t)=P_t*f(x)$ is the Poisson integral of $f$, where $P_t(x)=c_n\frac{t}{(t^2+|x|^2)^{{(n+1)}/2}}$ denotes the Poisson kernel in ${\mathbb R}^{n+1}_+$. Then we define the classical square function (Lusin area integral) $S(f)$ by (see \cite{stein2} and \cite{stein})
\begin{equation*}
S(f)(x)=\bigg(\iint_{\Gamma(x)}\big|\nabla u(y,t)\big|^2t^{1-n}\,dydt\bigg)^{1/2},
\end{equation*}
where $\Gamma(x)$ denotes the usual cone of aperture one:
\begin{equation*}
\Gamma(x)=\big\{(y,t)\in{\mathbb R}^{n+1}_+:|x-y|<t\big\}
\end{equation*}
and
\begin{equation*}
\big|\nabla u(y,t)\big|=\left|\frac{\partial u}{\partial t}\right|^2+\sum_{j=1}^n\left|\frac{\partial u}{\partial y_j}\right|^2.
\end{equation*}
Similarly, we can define a cone of aperture $\beta$ for any $\beta>0$:
\begin{equation*}
\Gamma_\beta(x)=\big\{(y,t)\in{\mathbb R}^{n+1}_+:|x-y|<\beta t\big\},
\end{equation*}
and corresponding square function
\begin{equation*}
S_\beta(f)(x)=\bigg(\iint_{\Gamma_\beta(x)}\big|\nabla u(y,t)\big|^2t^{1-n}\,dydt\bigg)^{1/2}.
\end{equation*}
The Littlewood-Paley $g$-function (could be viewed as a ``zero-aperture" version of $S(f)$) and the $g^*_\lambda$-function (could be viewed as an ``infinite aperture" version of $S(f)$) are defined respectively by
\begin{equation*}
g(f)(x)=\bigg(\int_0^\infty\big|\nabla u(x,t)\big|^2 t\,dt\bigg)^{1/2}
\end{equation*}
and
\begin{equation*}
g^*_\lambda(f)(x)=\left(\iint_{{\mathbb R}^{n+1}_+}\bigg(\frac t{t+|x-y|}\bigg)^{\lambda n}\big|\nabla u(y,t)\big|^2 t^{1-n}\,dydt\right)^{1/2}, \quad \lambda>1.
\end{equation*}

The modern (real-variable) variant of $S_\beta(f)$ can be defined in the following way (here we drop the subscript $\beta$ if $\beta=1$). Let $\psi\in C^\infty(\mathbb R^n)$ be real, radial, have support contained in $\{x:|x|\le1\}$, and $\int_{\mathbb R^n}\psi(x)\,dx=0$. The continuous square function $S_{\psi,\beta}(f)$ is defined by (see, for example, \cite{chang} and \cite{chanillo})
\begin{equation*}
S_{\psi,\beta}(f)(x)=\bigg(\iint_{\Gamma_\beta(x)}\big|f*\psi_t(y)\big|^2\frac{dydt}{t^{n+1}}\bigg)^{1/2}.
\end{equation*}

In 2007, Wilson \cite{wilson1} introduced a new square function called intrinsic square function which is universal in a sense (see also \cite{wilson2}). This function is independent of any particular kernel $\psi$, and it dominates pointwise all the above-defined square functions. On the other hand, it is not essentially larger than any particular $S_{\psi,\beta}(f)$. For $0<\alpha\le1$, let ${\mathcal C}_\alpha$ be the family of functions $\varphi$ defined on $\mathbb R^n$ such that $\varphi$ has support containing in $\{x\in\mathbb R^n: |x|\le1\}$, $\int_{\mathbb R^n}\varphi(x)\,dx=0$, and for all $x, x'\in \mathbb R^n$,
\begin{equation*}
|\varphi(x)-\varphi(x')|\le|x-x'|^\alpha.
\end{equation*}
For $(y,t)\in {\mathbb R}^{n+1}_+$ and $f\in L^1_{{loc}}(\mathbb R^n)$, we set
\begin{equation*}
A_\alpha(f)(y,t)=\sup_{\varphi\in{\mathcal C}_\alpha}\big|f*\varphi_t(y)\big|=\sup_{\varphi\in{\mathcal C}_\alpha}\bigg|\int_{\mathbb R^n}\varphi_t(y-z)f(z)\,dz\bigg|.
\end{equation*}
Then we define the intrinsic square function of $f$ (of order $\alpha$) by the formula
\begin{equation*}
\mathcal S_\alpha(f)(x)=\left(\iint_{\Gamma(x)}\Big(A_\alpha(f)(y,t)\Big)^2\frac{dydt}{t^{n+1}}\right)^{1/2}.
\end{equation*}
We can also define varying-aperture versions of $\mathcal S_\alpha(f)$ by the formula
\begin{equation*}
\mathcal S_{\alpha,\beta}(f)(x)=\left(\iint_{\Gamma_\beta(x)}\Big(A_\alpha(f)(y,t)\Big)^2\frac{dydt}{t^{n+1}}\right)^{1/2}.
\end{equation*}
The intrinsic Littlewood-Paley $g$-function and the intrinsic $g^*_\lambda$-function will be given respectively by
\begin{equation*}
g_\alpha(f)(x)=\left(\int_0^\infty\Big(A_\alpha(f)(x,t)\Big)^2\frac{dt}{t}\right)^{1/2}
\end{equation*}
and
\begin{equation*}
g^*_{\lambda,\alpha}(f)(x)=\left(\iint_{{\mathbb R}^{n+1}_+}\left(\frac t{t+|x-y|}\right)^{\lambda n}\Big(A_\alpha(f)(y,t)\Big)^2\frac{dydt}{t^{n+1}}\right)^{1/2}, \quad \lambda>1.
\end{equation*}

In \cite{wilson1} and \cite{wilson2}, Wilson has established the following theorems.

\newtheorem*{thma}{Theorem A}

\begin{thma}
Let $0<\alpha\le1$, $1<p<\infty$ and $w\in A_p (\mbox{Muckenhoupt weight class})$. Then there exists a constant $C>0$ independent of $f$ such that
\begin{equation*}
\|\mathcal S_\alpha(f)\|_{L^p_w}\le C \|f\|_{L^p_w}.
\end{equation*}
\end{thma}

\newtheorem*{thmb}{Theorem B}

\begin{thmb}
Let $0<\alpha\le1$ and $p=1$. Then for any given weight function $w$ and $\lambda>0$, there exists a constant $C>0$ independent of $f$ and $\lambda$ such that
\begin{equation*}
w\big(\big\{x\in\mathbb R^n:\mathcal S_\alpha(f)(x)>\lambda\big\}\big)\le\frac{C}{\lambda}\int_{\mathbb R^n}|f(x)|Mw(x)\,dx,
\end{equation*}
where $M$ denotes the standard Hardy-Littlewood maximal operator.
\end{thmb}

Moreover, in \cite{lerner}, Lerner showed sharp $L^p_w$ norm inequalities for the intrinsic square functions in terms of the $A_p$ characteristic constant of $w$ for all $1<p<\infty$. For further discussions about the boundedness of intrinsic square functions on some other weighted spaces, we refer the reader to \cite{huang,wang5,wang6,wang3,wang4}.

On the other hand, the classical Morrey spaces $\mathcal L^{p,\lambda}$ were first introduced by Morrey in \cite{morrey} to study the
local behavior of solutions to second order elliptic partial differential equations. Since then, these spaces play an important role in studying the regularity of solutions to partial differential equations. For the boundedness of the
Hardy-Littlewood maximal operator, the fractional integral operator and the Calder\'on-Zygmund singular integral
operator on these spaces, we refer the reader to \cite{adams,chiarenza,peetre}. For the properties and applications
of classical Morrey spaces, see \cite{fan,fazio1,fazio2} and the references therein.

In 2009, Komori and Shirai \cite{komori} first defined the weighted Morrey spaces $L^{p,\kappa}(w)$ which could be
viewed as an extension of weighted Lebesgue spaces, and studied the boundedness of the above classical operators on
these weighted spaces. Recently, in \cite{wang1,wang5,wang7,wang2}, we have established the continuity properties of some other operators on the weighted Morrey spaces $L^{p,\kappa}(w)$.

In \cite{wang5}, we studied the boundedness properties of intrinsic square functions on the weighted Morrey spaces $L^{p,\kappa}(w)$ for all $1<p<\infty$, $0<\kappa<1$ and $w\in A_p$. As a continuation of this work, the main purpose of this paper is to investigate their weak type estimates on the weighted Morrey spaces $L^{1,\kappa}(w)$ when $0<\kappa<1$ and $w\in A_1$. Our main results in the paper are formulated as follows.

\newtheorem{theorem}{Theorem}[section]

\begin{theorem}
Let $0<\alpha\le1$, $0<\kappa<1$ and $w\in A_1$. Then there is a
constant $C>0$ independent of $f$ such that
\begin{equation*}
\big\|\mathcal S_\alpha(f)\big\|_{WL^{1,\kappa}(w)}\le C\|f\|_{L^{1,\kappa}(w)}.
\end{equation*}
\end{theorem}

\begin{theorem}
Let $0<\alpha\le1$, $0<\kappa<1$ and $w\in A_1$. If $\lambda>{(3n+2\alpha)}/n$, then there is a
constant $C>0$ independent of $f$ such that
\begin{equation*}
\big\|g^*_{\lambda,\alpha}(f)\big\|_{WL^{1,\kappa}(w)}\le C\|f\|_{L^{1,\kappa}(w)}.
\end{equation*}
\end{theorem}

In \cite{wilson1}, Wilson also showed that for any $0<\alpha\le1$, the functions $\mathcal S_\alpha(f)(x)$ and $g_\alpha(f)(x)$ are pointwise comparable. Thus, as a direct consequence of Theorem 1.1, we obtain the following

\newtheorem{corollary}[theorem]{Corollary}

\begin{corollary}
Let $0<\alpha\le1$, $0<\kappa<1$ and $w\in A_1$. Then there is a
constant $C>0$ independent of $f$ such that
\begin{equation*}
\big\|g_\alpha(f)\big\|_{WL^{1,\kappa}(w)}\le C\|f\|_{L^{1,\kappa}(w)}.
\end{equation*}
\end{corollary}

\section{Notations and definitions}

The classical $A_p$ weight theory was first introduced by Muckenhoupt in the study of weighted $L^p$ boundedness of Hardy-Littlewood maximal functions in \cite{muckenhoupt1}. A weight $w$ is a nonnegative, locally integrable function on $\mathbb R^n$, $B=B(x_0,r_B)$ denotes the ball with the center $x_0$ and radius $r_B$. Given a ball $B$ and $\lambda>0$, $\lambda B$ denotes the ball with the same center as $B$ whose radius is $\lambda$ times that of $B$. For a given weight function $w$ and a measurable set $E$ in $\mathbb R^n$, we also denote the Lebesgue measure of $E$ by $|E|$ and the weighted measure of $E$ by $w(E)$, where $w(E)=\int_E w(x)\,dx$. We say that $w$ is in the Muckenhoupt class $A_p$ with $1<p<\infty$, if
\begin{equation*}
\left(\frac1{|B|}\int_B w(x)\,dx\right)\left(\frac1{|B|}\int_B w(x)^{-1/{(p-1)}}\,dx\right)^{p-1}\le C \quad\mbox{for every ball}\; B\subseteq \mathbb
R^n,
\end{equation*}
where $C$ is a positive constant which is independent of the choice of $B$.
For the endpoint case $p=1$, $w\in A_1$, if
\begin{equation*}
\frac1{|B|}\int_B w(x)\,dx\le C\cdot\underset{x\in B}{\mbox{ess\,inf}}\,w(x)\quad\mbox{for every ball}\;B\subseteq\mathbb R^n.
\end{equation*}

A weight function $w$ is said to belong to the reverse H\"older class $RH_r$ if there exist two constants $r>1$ and $C>0$ such that the following reverse H\"older inequality
\begin{equation*}
\left(\frac{1}{|B|}\int_B w(x)^r\,dx\right)^{1/r}\le C\left(\frac{1}{|B|}\int_B w(x)\,dx\right)
\end{equation*}
holds for every ball $B$ in $\mathbb R^n$.

It is well known that if $w\in A_p$ with $p=1$, then $w\in A_q$ for all $q>1$. If $w\in A_p$ with $1\le p<\infty$, then there exists $r>1$ such that $w\in RH_r$. We state the following results that will be used in the sequel.
\newtheorem{lemma}[theorem]{Lemma}

\begin{lemma}[\cite{garcia2}]
Let $w\in A_p$ with $1\le p<\infty$. Then, for any ball $B$, there exists an absolute constant $C>0$ such that
\begin{equation*}
w(2B)\le C\,w(B).
\end{equation*}
In general, for any $\lambda>1$, we have
\begin{equation*}
w(\lambda B)\le C\cdot\lambda^{np}w(B),
\end{equation*}
where $C$ does not depend on $B$ nor on $\lambda$.
\end{lemma}

\begin{lemma}[\cite{gundy}]
Let $w\in RH_r$ with $r>1$. Then there exists a constant $C>0$ such that
\begin{equation*}
\frac{w(E)}{w(B)}\le C\left(\frac{|E|}{|B|}\right)^{(r-1)/r}
\end{equation*}
for any measurable subset $E$ of a ball $B$.
\end{lemma}

\begin{lemma}[\cite{garcia2}]
Let $w\in A_q$ with $q>1$. Then, for all $R>0$, there exists a constant $C>0$ independent of $R$ such that
\begin{equation*}
\int_{|x|\ge R}\frac{w(x)}{|x|^{nq}}\,dx\le C\cdot R^{-nq}w\big(Q(0,2R)\big),
\end{equation*}
where $Q=Q(x_0,\ell)$ denotes the cube centered at $x_0$ with side length $\ell$ and all cubes are
assumed to have their sides parallel to the coordinate axes.
\end{lemma}

Given a weight function $w$ on $\mathbb R^n$, for $1\le p<\infty$, we denote by $L^p_w(\mathbb R^n)$ the weighted space of all functions $f$ satisfying
\begin{equation*}
\|f\|_{L^p_w}=\bigg(\int_{\mathbb R^n}|f(x)|^pw(x)\,dx\bigg)^{1/p}<\infty.
\end{equation*}

\newtheorem{defn}[theorem]{Definition}

\begin{defn}[\cite{komori}]
Let $1\le p<\infty$, $0<\kappa<1$ and $w$ be a weight function. Then the weighted Morrey space is defined by
\begin{equation*}
L^{p,\kappa}(w)=\big\{f\in L^p_{loc}(w):\|f\|_{L^{p,\kappa}(w)}<\infty\big\},
\end{equation*}
where
\begin{equation*}
\|f\|_{L^{p,\kappa}(w)}=\sup_B\left(\frac{1}{w(B)^\kappa}\int_B|f(x)|^pw(x)\,dx\right)^{1/p}
\end{equation*}
and the supremum is taken over all balls $B$ in $\mathbb R^n$.
\end{defn}

We also denote by $WL^{p,\kappa}(w)$ the weighted weak Morrey spaces of all measurable functions $f$ satisfying
\begin{equation*}
\|f\|_{WL^{p,\kappa}(w)}=\sup_B\sup_{\lambda>0}\frac{1}{w(B)^{\kappa/p}}\lambda\cdot w\big(\big\{x\in B:|f(x)|>\lambda\big\}\big)^{1/p}<\infty.
\end{equation*}

Throughout this paper, the letter $C$ always denote a positive constant independent of the main parameters involved, but it may be different from line to line.

\section{Proof of Theorem 1.1}

First we note that if $w\in A_1$, then $M(w)(x)\le C\cdot w(x)$ for a.e. $x\in\mathbb R^n$ by the definition of $A_1$ weights. Hence, as a straightforward consequence of Theorem B, we obtain

\begin{theorem}
Let $0<\alpha\le1$ and $w\in A_1$. Then for any given $\lambda>0$, there exists a constant $C>0$ independent of $f$ and $\lambda$ such that
\begin{equation*}
w\big(\big\{x\in\mathbb R^n:|\mathcal S_\alpha(f)(x)|>\lambda\big\}\big)\le\frac{C}{\lambda}\int_{\mathbb R^n}|f(x)|w(x)\,dx.
\end{equation*}
\end{theorem}

We are now ready to prove Theorem 1.1.

\begin{proof}[Proof of Theorem 1.1]
Let $f\in L^{1,\kappa}(w)$. Fix a ball $B=B(x_0,r_B)\subseteq\mathbb R^n$ and decompose $f=f_1+f_2$, where $f_1=f\chi_{_{2B}}$, $\chi_{_{2B}}$ denotes the characteristic function of $2B$. Since $\mathcal S_\alpha$ ($0<\alpha\le1$) is a sublinear operator, then for any given $\lambda>0$, we write
\begin{equation*}
\begin{split}
&w\big(\big\{x\in B:|\mathcal S_{\alpha}(f)(x)|>\lambda\big\}\big)\\
\le\,& w\big(\big\{x\in B:|\mathcal S_{\alpha}(f_1)(x)|>\lambda/2\big\}\big)+w\big(\big\{x\in B:|\mathcal S_{\alpha}(f_2)(x)|>\lambda/2\big\}\big)\\
  =\,&I_1+I_2.
\end{split}
\end{equation*}
Lemma 2.1 and Theorem 3.1 yield
\begin{equation*}
\begin{split}
I_1&\le\frac{C}{\lambda}\int_{2B}|f(y)|w(y)\,dy\\
&\le\frac{C\cdot w(2B)^\kappa}{\lambda}\|f\|_{L^{1,\kappa}(w)}\\
&\le\frac{C\cdot w(B)^\kappa}{\lambda}\|f\|_{L^{1,\kappa}(w)}.
\end{split}
\end{equation*}
We now turn to estimate the other term $I_2$. For any $\varphi\in{\mathcal C}_\alpha$, $0<\alpha\le1$ and $(y,t)\in\Gamma(x)$, we have
\begin{align}
\big|f_2*\varphi_t(y)\big|&=\bigg|\int_{(2B)^c}\varphi_t(y-z)f(z)\,dz\bigg|\notag\\
&\le C\cdot t^{-n}\int_{(2B)^c\cap\{z:|y-z|\le t\}}|f(z)|\,dz\notag\\
&\le C\cdot t^{-n}\sum_{j=1}^\infty\int_{(2^{j+1}B\backslash 2^{j}B)\cap\{z:|y-z|\le t\}}|f(z)|\,dz.
\end{align}
For any $x\in B$, $(y,t)\in\Gamma(x)$ and $z\in\big(2^{j+1}B\backslash 2^{j}B\big)\cap B(y,t)$, then by a direct computation, we can easily see that
\begin{equation*}
2t\ge |x-y|+|y-z|\ge|x-z|\ge|z-x_0|-|x-x_0|\ge 2^{j-1}r_B.
\end{equation*}
Thus, by using the above inequality (1) and Minkowski's inequality, we deduce
\begin{align}
\big|\mathcal S_\alpha(f_2)(x)\big|&=\left(\iint_{\Gamma(x)}\Big(\sup_{\varphi\in{\mathcal C}_\alpha}\big|f_2*\varphi_t(y)\big|\Big)^2\frac{dydt}{t^{n+1}}\right)^{1/2}\notag\\
&\le C\left(\int_{2^{j-2}r_B}^\infty\int_{|x-y|<t}\bigg|t^{-n}\sum_{j=1}^\infty\int_{2^{j+1}B\backslash 2^{j}B}|f(z)|\,dz\bigg|^2\frac{dydt}{t^{n+1}}\right)^{1/2}\notag\\
&\le C\bigg(\sum_{j=1}^\infty\int_{2^{j+1}B\backslash 2^{j}B}|f(z)|\,dz\bigg)\left(\int_{2^{j-2}r_B}^\infty\frac{dt}{t^{2n+1}}\right)^{1/2}\notag\\
&\le C\sum_{j=1}^\infty\frac{1}{|2^{j+1}B|}\int_{2^{j+1}B}|f(z)|\,dz.
\end{align}
It follows directly from the $A_1$ condition that
\begin{align}
\sum_{j=1}^\infty\frac{1}{|2^{j+1}B|}\int_{2^{j+1}B}|f(z)|\,dz
&\le C\sum_{j=1}^\infty\frac{1}{w(2^{j+1}B)}\underset{z\in 2^{j+1}B}{\mbox{ess\,inf}}\,w(z)\int_{2^{j+1}B}|f(z)|\,dz\notag\\
&\le C\sum_{j=1}^\infty\frac{1}{w(2^{j+1}B)}\int_{2^{j+1}B}|f(z)|w(z)\,dz\notag\\
&\le C\|f\|_{L^{1,\kappa}(w)}\cdot\frac{1}{w(B)^{1-\kappa}}
\sum_{j=1}^\infty\frac{w(B)^{1-\kappa}}{w(2^{j+1}B)^{1-\kappa}}.
\end{align}
Since $w\in A_1$, then there exists a number $r>1$ such that $w\in RH_r$. Consequently, by Lemma 2.2, we obtain
\begin{equation}
\frac{w(B)}{w(2^{j+1}B)}\le C\left(\frac{|B|}{|2^{j+1}B|}\right)^{(r-1)/{r}}.
\end{equation}
Hence, for any $x\in B$,
\begin{align}
\big|\mathcal S_\alpha(f_2)(x)\big|&\le C\|f\|_{L^{1,\kappa}(w)}\cdot\frac{1}{w(B)^{1-\kappa}}
\sum_{j=1}^\infty\left(\frac{1}{2^{jn}}\right)^{(1-\kappa)(r-1)/{r}}\notag\\
&\le C\|f\|_{L^{1,\kappa}(w)}\cdot\frac{1}{w(B)^{1-\kappa}},
\end{align}
where in the last inequality we have used the fact that $(1-\kappa)(r-1)/{r}>0$. If $\big\{x\in B:|\mathcal S_{\alpha}(f_2)(x)|>\lambda/2\big\}=\O$, then the inequality
\begin{equation*}
I_2\le \frac{C\cdot w(B)^\kappa}{\lambda}\|f\|_{L^{1,\kappa}(w)}
\end{equation*}
holds trivially. Now we may suppose that $\big\{x\in B:|\mathcal S_{\alpha}(f_2)(x)|>\lambda/2\big\}\neq\O$, then by the pointwise inequality (5), we have
\begin{equation*}
\lambda\le C\|f\|_{L^{1,\kappa}(w)}\cdot\frac{1}{w(B)^{1-\kappa}},
\end{equation*}
which is equivalent to
$$w(B)\le\frac{C\cdot w(B)^\kappa}{\lambda}\|f\|_{L^{1,\kappa}(w)}.$$
Therefore
\begin{equation*}
I_2\le w(B)\le\frac{C\cdot w(B)^\kappa}{\lambda}\|f\|_{L^{1,\kappa}(w)}.
\end{equation*}
Summing up the above estimates for $I_1$ and $I_2$, and then taking the supremum over all balls $B\subseteq\mathbb R^n$ and all $\lambda>0$, we complete the proof of Theorem 1.1.
\end{proof}

\section{Proof of Theorem 1.2}

Before proving the main theorem in this section, let us first establish the following results.
\begin{lemma}
Let $0<\alpha\le1$ and $w\in A_1$. Then for any $j\in\mathbb Z_+$, we have
\begin{equation*}
\big\|\mathcal S_{\alpha,2^j}(f)\big\|_{L^2_w}\le C\cdot2^{{jn}/2}\big\|\mathcal S_\alpha(f)\big\|_{L^2_w}.
\end{equation*}
\end{lemma}

\begin{proof}
Since $w\in A_1$, then by Lemma 2.1, we know that for any $(y,t)\in{\mathbb R}^{n+1}_+$,
\begin{equation*}
w\big(B(y,2^jt)\big)=w\big(2^jB(y,t)\big)\le C\cdot2^{jn}w\big(B(y,t)\big)\quad j=1,2,\ldots.
\end{equation*}
Therefore
\begin{equation*}
\begin{split}
\big\|\mathcal S_{\alpha,2^j}(f)\big\|_{L^2_w}^2&=\int_{\mathbb R^n}\bigg(\iint_{{\mathbb R}^{n+1}_+}\Big(A_\alpha(f)(y,t)\Big)^2\chi_{|x-y|<2^j t}\frac{dydt}{t^{n+1}}\bigg)w(x)\,dx\\
&=\iint_{{\mathbb R}^{n+1}_+}\Big(\int_{|x-y|<2^j t}w(x)\,dx\Big)\Big(A_\alpha(f)(y,t)\Big)^2\frac{dydt}{t^{n+1}}\\
&\le C\cdot2^{jn}\iint_{{\mathbb R}^{n+1}_+}\Big(\int_{|x-y|<t}w(x)\,dx\Big)\Big(A_\alpha(f)(y,t)\Big)^2\frac{dydt}{t^{n+1}}\\
&=C\cdot 2^{jn}\big\|\mathcal S_\alpha(f)\big\|_{L^2_w}^2.
\end{split}
\end{equation*}
Taking square-roots on both sides of the above inequality, we are done.
\end{proof}

\begin{theorem}
Let $0<\alpha\le1$, $w\in A_1$ and $\lambda>{(3n+2\alpha)}/n$. Then for any given $\sigma>0$, there exists a constant $C>0$ independent of $f$ and $\sigma$ such that
\begin{equation*}
w\Big(\Big\{x\in\mathbb R^n:\big|g^*_{\lambda,\alpha}(f)(x)\big|>\sigma\Big\}\Big)\le\frac{C}{\sigma}\int_{\mathbb R^n}|f(x)|w(x)\,dx.
\end{equation*}
\end{theorem}

\begin{proof}
First, from the definition of $g^*_{\lambda,\alpha}$, we readily see that
\begin{equation*}
\begin{split}
g^*_{\lambda,\alpha}(f)(x)^2=&\iint_{\mathbb R^{n+1}_+}\left(\frac{t}{t+|x-y|}\right)^{\lambda n}\Big(A_\alpha(f)(y,t)\Big)^2\frac{dydt}{t^{n+1}}\notag\\
=&\int_0^\infty\int_{|x-y|<t}\left(\frac{t}{t+|x-y|}\right)^{\lambda n}\Big(A_\alpha(f)(y,t)\Big)^2\frac{dydt}{t^{n+1}}\notag\\
\end{split}
\end{equation*}
\begin{align}
&+\sum_{j=1}^\infty\int_0^\infty\int_{2^{j-1}t\le|x-y|<2^jt}\left(\frac{t}{t+|x-y|}\right)^{\lambda n}\Big(A_\alpha(f)(y,t)\Big)^2\frac{dydt}{t^{n+1}}\notag\\
\le&\, C\bigg[\mathcal S_\alpha(f)(x)^2+\sum_{j=1}^\infty 2^{-j\lambda n}\mathcal S_{\alpha,2^j}(f)(x)^2\bigg].
\end{align}
Then for any given $\sigma>0$, it follows from the above inequality (6) that
\begin{equation*}
\begin{split}
&w\Big(\Big\{x\in \mathbb R^n:\big|g^*_{\lambda,\alpha}(f)(x)\big|>\sigma\Big\}\Big)\\
\le&\,w\Big(\Big\{x\in \mathbb R^n:|\mathcal S_{\alpha}(f)(x)|>\sigma/2\Big\}\Big)+w\Big(\Big\{x\in \mathbb R^n:\Big|\sum_{j=1}^\infty 2^{-j\lambda n/2}\mathcal S_{\alpha,2^j}(f)(x)\Big|>\sigma/2\Big\}\Big)\\
=&\,\mbox{\upshape I+II}.
\end{split}
\end{equation*}
Using Theorem 3.1, we can get
\begin{equation*}
\mbox{I}\le\frac{C}{\sigma}\int_{\mathbb R^n}|f(x)|w(x)\,dx.
\end{equation*}
In order to estimate the term II, for any fixed $\sigma>0$, we apply the Calder\'on-Zygmund decomposition of $f$ at height $\sigma$ to obtain a sequence of disjoint non-overlapping dyadic cubes $\{Q_i\}$ such that the following property hold (see \cite{stein})
\begin{equation}
\sigma<\frac{1}{|Q_i|}\int_{Q_i}|f(y)|\,dy< 2^n\sigma.
\end{equation}
Setting $E=\bigcup_i Q_i$. Now we define two functions $g$ and $b$ as follows:
\begin{equation*}
g(x)=
\begin{cases}
f(x) &  \mbox{if}\;\; x\in E^c,\\
\frac{1}{|Q_i|}\int_{Q_i}|f(y)|\,dy    &  \mbox{if}\;\; x\in Q_i,
\end{cases}
\end{equation*}
and
\begin{equation*}
b(x)=f(x)-g(x)=\sum_i b_i(x),
\end{equation*}
where $b_i(x)=b(x)\chi_{Q_i}(x)$. Then we have
\begin{equation}
|g(x)|\le C\cdot\sigma, \quad \mbox{a.e. }\, x\in\mathbb R^n
\end{equation}
and
\begin{equation}
f(x)=g(x)+b(x).
\end{equation}
Obviously, $supp\,b_i\subseteq Q_i$, $\int_{Q_i}b_i(x)\,dx=0$ and $\|b_i\|_{L^1}\le 2\int_{Q_i}|f(x)|\,dx$ by our construction. Now for $j=1,2,\ldots$, since $\mathcal S_{\alpha,2^j}(f)(x)\le \mathcal S_{\alpha,2^j}(g)(x)+\mathcal S_{\alpha,2^j}(b)(x)$ by (9), then it follows that
\begin{equation*}
\begin{split}
\mbox{II}\le&\, w\Big(\Big\{x\in \mathbb R^n:\Big|\sum_{j=1}^\infty 2^{-j\lambda n/2}\mathcal S_{\alpha,2^j}(g)(x)\Big|>\sigma/4\Big\}\Big)\\
&+w\Big(\Big\{x\in \mathbb R^n:\Big|\sum_{j=1}^\infty 2^{-j\lambda n/2}\mathcal S_{\alpha,2^j}(b)(x)\Big|>\sigma/4\Big\}\Big)\\
=&\,\mbox{\upshape III+IV}.
\end{split}
\end{equation*}
Observe that $w\in A_1\subset A_2$ and $\lambda>1$. Applying Chebyshev's inequality, Minkowski's inequality, Lemma 4.1 and Theorem A, we obtain
\begin{equation*}
\begin{split}
\mbox{\upshape III}&\le\frac{C}{\sigma^2}\bigg\|\sum_{j=1}^\infty 2^{-j\lambda n/2}\mathcal S_{\alpha,2^j}(g)\bigg\|^2_{L^2_w}\\
&\le\frac{C}{\sigma^2}\Bigg(\sum_{j=1}^\infty 2^{-j\lambda n/2}\cdot 2^{jn/2}\big\|\mathcal S_\alpha(g)\big\|_{L^2_w}\Bigg)^2\\
&\le\frac{C}{\sigma^2}\Bigg(\sum_{j=1}^\infty 2^{-j\lambda n/2}\cdot 2^{jn/2}\|g\|_{L^2_w}\Bigg)^2\\
&\le\frac{C}{\sigma^2}\cdot\big\|g\big\|^2_{L^2_w}.
\end{split}
\end{equation*}
Moreover, by the inequality (8) and the $A_1$ condition, we deduce that
\begin{equation*}
\begin{split}
\big\|g\big\|^2_{L^2_w}&\le C\cdot\sigma\int_{\mathbb R^n}|g(x)|w(x)\,dx\\
&\le C\cdot\sigma\left(\int_{E^c}|f(x)|w(x)\,dx+\int_{\bigcup_i Q_i}|g(x)|w(x)\,dx\right)\\
&\le C\cdot\sigma\left(\int_{\mathbb R^n}|f(x)|w(x)\,dx+\sum_i\frac{w(Q_i)}{|Q_i|}\int_{Q_i}|f(y)|\,dy\right)\\
&\le C\cdot\sigma\left(\int_{\mathbb R^n}|f(x)|w(x)\,dx+\sum_i\underset{y\in Q_i}{\mbox{ess\,inf}}\,w(y)\int_{Q_i}|f(y)|\,dy\right)\\
&\le C\cdot\sigma\left(\int_{\mathbb R^n}|f(x)|w(x)\,dx+\int_{\bigcup_i Q_i}|f(y)|w(y)\,dy\right)\\
&\le C\cdot\sigma\int_{\mathbb R^n}|f(x)|w(x)\,dx.
\end{split}
\end{equation*}
So we have
\begin{equation*}
\mbox{III}\le\frac{C}{\sigma}\int_{\mathbb R^n}|f(x)|w(x)\,dx.
\end{equation*}
To deal with the last term IV, let $Q_i^*=2\sqrt n Q_i$ be the cube concentric with $Q_i$ such that $\ell(Q_i^*)=(2\sqrt n)\ell(Q_i)$. Then we can further decompose IV as follows.
\begin{equation*}
\begin{split}
\mbox{\upshape IV}\le&\,w\Big(\Big\{x\in \bigcup_i Q_i^*:\Big|\sum_{j=1}^\infty 2^{-j\lambda n/2}\mathcal S_{\alpha,2^j}(b)(x)\Big|>\sigma/4\Big\}\Big)\\
&+w\Big(\Big\{x\notin \bigcup_i Q_i^*:\Big|\sum_{j=1}^\infty 2^{-j\lambda n/2}\mathcal S_{\alpha,2^j}(b)(x)\Big|>\sigma/4\Big\}\Big)\\
=&\,\mbox{\upshape IV}^{(1)}+\mbox{\upshape IV}^{(2)}.
\end{split}
\end{equation*}
Since $w\in A_1$, then by Lemma 2.1, we can get
\begin{equation*}
\mbox{\upshape IV}^{(1)}\le\sum_i w\big(Q_i^*\big)\le C\sum_i w(Q_i).
\end{equation*}
Furthermore, it follows from the inequality (7) and the $A_1$ condition that
\begin{equation*}
\begin{split}
\mbox{\upshape IV}^{(1)}&\le C\sum_i\frac{1}{\sigma}\cdot\underset{y\in Q_i}{\mbox{ess\,inf}}\,w(y)\int_{Q_i}|f(y)|\,dy\\
&\le \frac{C}{\sigma}\sum_i\int_{Q_i}|f(y)|w(y)\,dy\\
&\le \frac{C}{\sigma}\int_{\bigcup_i Q_i}|f(y)|w(y)\,dy\\
&\le \frac{C}{\sigma}\int_{\mathbb R^n}|f(y)|w(y)\,dy.
\end{split}
\end{equation*}
Thus, in order to finish our proof, we need only to prove that
\begin{equation*}
\mbox{\upshape IV}^{(2)}\le\frac{C}{\sigma}\int_{\mathbb R^n}|f(x)|w(x)\,dx.
\end{equation*}
Denote the center of $Q_i$ by $c_i$. For any $\varphi\in{\mathcal C}_\alpha$, $0<\alpha\le1$, by the cancellation condition of $b_i$, we have that for any $(y,t)\in\Gamma_{2^j}(x)$,
\begin{align}
\big|(b_i*\varphi_t)(y)\big|&=\left|\int_{Q_i}\big[\varphi_t(y-z)-\varphi_t(y-c_i)\big]b_i(z)\,dz\right|\notag\\
&\le\int_{Q_i\cap\{z:|z-y|\le t\}}\frac{|z-c_i|^\alpha}{t^{n+\alpha}}|b_i(z)|\,dz\notag\\
&\le C\cdot\frac{\ell(Q_i)^{\alpha}}{t^{n+\alpha}}\int_{Q_i\cap\{z:|z-y|\le t\}}|b_i(z)|\,dz.
\end{align}
In addition, for any $z\in Q_i$ and $x\in (Q^*_i)^c$, we have $|z-c_i|<\frac{|x-c_i|}{2}$. Thus, for all $(y,t)\in\Gamma_{2^j}(x)$ and $|z-y|\le t$ with $z\in Q_i$, we can deduce that
\begin{equation}
t+2^jt\ge|x-y|+|y-z|\ge|x-z|\ge|x-c_i|-|z-c_i|\ge\frac{|x-c_i|}{2}.
\end{equation}
Hence, for any $x\in (Q^*_i)^c$, by using the above inequalities (10) and (11), we obtain
\begin{equation*}
\begin{split}
\big|\mathcal S_{\alpha,2^j}(b_i)(x)\big|&=\left(\iint_{\Gamma_{2^j}(x)}\Big(\sup_{\varphi\in{\mathcal C}_\alpha}\big|(\varphi_t*{b_i})(y)\big|\Big)^2\frac{dydt}{t^{n+1}}\right)^{1/2}\\
&\le C\cdot\ell(Q_i)^{\alpha}\bigg(\int_{Q_i}|b_i(z)|\,dz\bigg)\left(\int_{\frac{|x-c_i|}{2^{j+2}}}^\infty
\int_{|y-x|<2^jt}\frac{dydt}{t^{2(n+\alpha)+n+1}}\right)^{1/2}\\
\end{split}
\end{equation*}
\begin{equation*}
\begin{split}
&\le C\cdot2^{{jn}/2}\ell(Q_i)^{\alpha}\bigg(\int_{Q_i}|b_i(z)|\,dz\bigg)
\left(\int_{\frac{|x-c_i|}{2^{j+2}}}^\infty\frac{dt}{t^{2(n+\alpha)+1}}\right)^{1/2}\\
&\le C\cdot2^{j(3n+2\alpha)/2}\frac{\ell(Q_i)^{\alpha}}{|x-c_i|^{n+\alpha}}\bigg(\int_{Q_i}|f(z)|\,dz\bigg).
\end{split}
\end{equation*}
This estimate together with the Chebyshev's inequality yields
\begin{equation*}
\begin{split}
\mbox{\upshape IV}^{(2)}&\le\frac{4}{\sigma}\int_{\mathbb R^n\backslash\bigcup_i Q_i^*}\bigg|\sum_{j=1}^\infty 2^{-j\lambda n/2}\mathcal S_{\alpha,2^j}(b)(x)\bigg|w(x)\,dx\\
&\le\frac{4}{\sigma}\sum_{j=1}^\infty 2^{-j\lambda n/2}\sum_i\left(\int_{(Q_i^*)^c}\mathcal S_{\alpha,2^j}(b_i)(x)w(x)\,dx\right)\\
&\le\frac{C}{\sigma}\Bigg(\sum_{j=1}^\infty 2^{-j\lambda n/2}\cdot2^{j(3n+2\alpha)/2}\Bigg)
\left(\sum_i\ell(Q_i)^{\alpha}\int_{Q_i}|f(z)|\,dz\times\int_{(Q_i^*)^c}\frac{w(x)}{|x-c_i|^{n+\alpha}}dx\right)\\
&\le\frac{C}{\sigma}
\left(\sum_i\ell(Q_i)^{\alpha}\int_{Q_i}|f(z)|\,dz\times\int_{(Q_i^*)^c}\frac{w(x)}{|x-c_i|^{n+\alpha}}dx\right),
\end{split}
\end{equation*}
where the last inequality holds under our assumption of $\lambda>{(3n+2\alpha)}/n$. On the other hand, since $w\in A_1\subset A_{1+\alpha/n}$, then by Lemmas 2.3 and 2.1, we get
\begin{equation*}
\begin{split}
\int_{(Q_i^*)^c}\frac{w(x)}{|x-c_i|^{n+\alpha}}dx&=\int_{|y|\ge \sqrt n\ell(Q_i)} \frac{w_1(y)}{|y|^{n+\alpha}}dy\\
&\le C\cdot\ell(Q_i)^{-n-\alpha}w_1\big(Q(0,2\sqrt n\ell(Q_i))\big)\\
&=C\cdot\ell(Q_i)^{-n-\alpha}w\big(Q(c_i,2\sqrt n\ell(Q_i))\big)\\
&\le C\cdot\ell(Q_i)^{-n-\alpha}w(Q_i),
\end{split}
\end{equation*}
where $w_1(x)=w(x+c_i)$ is the translation of $w(x)$. It is obvious that $w_1\in A_1$ whenever $w\in A_1$. Hence, by using the $A_1$ condition again, we obtain
\begin{equation*}
\begin{split}
\mbox{\upshape IV}^{(2)}&\le\frac{C}{\sigma}\sum_i\frac{w(Q_i)}{|Q_i|}\int_{Q_i}|f(z)|\,dz\\
&\le\frac{C}{\sigma}\sum_i\underset{z\in Q_i}{\mbox{ess\,inf}}\,w(z)\int_{Q_i}|f(z)|\,dz\\
&\le\frac{C}{\sigma}\int_{\mathbb R^n}|f(z)|w(z)\,dz,
\end{split}
\end{equation*}
which is just our desired conclusion. Summarizing the estimates for I--IV derived above, we thus complete the proof of Theorem 4.2.
\end{proof}

We are now in a position to give the proof of Theorem 1.2.

\begin{proof}[Proof of Theorem 1.2]
Let $f\in L^{1,\kappa}(w)$. As in the proof of Theorem 1.1, we set $f=f_1+f_2$, where $f_1=f\chi_{_{2B}}$. Then for each fixed $\sigma>0$, we can write
\begin{equation*}
\begin{split}
&w\Big(\Big\{x\in B:\big|g^*_{\lambda,\alpha}(f)(x)\big|>\sigma\Big\}\Big)\\
\le\,& w\Big(\Big\{x\in B:\big|g^*_{\lambda,\alpha}(f_1)(x)\big|>\sigma/2\Big\}\Big)
+w\Big(\Big\{x\in B:\big|g^*_{\lambda,\alpha}(f_2)(x)\big|>\sigma/2\Big\}\Big)\\
  =\,&J_1+J_2.
\end{split}
\end{equation*}
Theorem 4.2 and Lemma 2.1 imply
\begin{equation*}
\begin{split}
J_1&\le\frac{C}{\sigma}\int_{2B}|f(y)|w(y)\,dy\\
&\le\frac{C\cdot w(2B)^\kappa}{\sigma}\|f\|_{L^{1,\kappa}(w)}\\
&\le\frac{C\cdot w(B)^\kappa}{\sigma}\|f\|_{L^{1,\kappa}(w)}.
\end{split}
\end{equation*}
We now turn to deal with the term $J_2$. Recall that in the proof of Theorem 1.1, we have already showed that for any $x\in B$,
\begin{equation}
\big|\mathcal S_\alpha(f_2)(x)\big|\le C\|f\|_{L^{1,\kappa}(w)}\cdot\frac{1}{w(B)^{1-\kappa}}.
\end{equation}
On the other hand, for any $x\in B$, $(y,t)\in\Gamma_{2^j}(x)$ and $z\in\big(2^{k+1}B\backslash 2^{k}B\big)\cap B(y,t)$, then by a simple calculation, we can easily deduce
\begin{equation*}
t+2^j t\ge |x-y|+|y-z|\ge|x-z|\ge|z-x_0|-|x-x_0|\ge 2^{k-1}r_B.
\end{equation*}
Hence, it follows from the previous inequality (1) and Minkowski's inequality that
\begin{equation*}
\begin{split}
\big|\mathcal S_{\alpha,2^j}(f_2)(x)\big|&=\left(\iint_{\Gamma_{2^j}(x)}\Big(\sup_{\varphi\in{\mathcal C}_\alpha}\big|f_2*\varphi_t(y)\big|\Big)^2\frac{dydt}{t^{n+1}}\right)^{1/2}\\
&\le C\left(\int_{2^{(k-2-j)}r_B}^\infty\int_{|x-y|<2^jt}\bigg|t^{-n}\sum_{k=1}^\infty\int_{2^{k+1}B\backslash 2^{k}B}|f(z)|\,dz\bigg|^2\frac{dydt}{t^{n+1}}\right)^{1/2}\\
&\le C\bigg(\sum_{k=1}^\infty\int_{2^{k+1}B\backslash 2^{k}B}|f(z)|\,dz\bigg)\bigg(\int_{2^{(k-2-j)}r_B}^\infty 2^{jn}\frac{dt}{t^{2n+1}}\bigg)^{1/2}\\
&\le C\cdot2^{{3jn}/2}\sum_{k=1}^\infty\frac{1}{|2^{k+1}B|}\int_{2^{k+1}B}|f(z)|\,dz.
\end{split}
\end{equation*}
Furthermore, by using the estimates (3) and (4), we can proceed as in (2) and get
\begin{equation}
\big|\mathcal S_{\alpha,2^j}(f_2)(x)\big|\le C\cdot2^{{3jn}/2}\|f\|_{L^{1,\kappa}(w)}\cdot\frac{1}{w(B)^{1-\kappa}}.
\end{equation}
Therefore, for any $x\in B$, by the inequalities (6), (12) and (13), we have
\begin{equation*}
\begin{split}
\big|g^*_{\lambda,\alpha}(f_2)(x)\big|&\le C\Bigg(\big|\mathcal S_\alpha(f_2)(x)\big|+\sum_{j=1}^\infty 2^{{-j\lambda n}/2}\big|\mathcal S_{\alpha,2^j}(f_2)(x)\big|\Bigg)\\
&\le C\|f\|_{L^{1,\kappa}(w)}\cdot\frac{1}{w(B)^{1-\kappa}}\Bigg(1+\sum_{j=1}^\infty2^{{-j\lambda n}/2}\cdot2^{{3jn}/2}\Bigg)\\
&\le C\|f\|_{L^{1,\kappa}(w)}\cdot\frac{1}{w(B)^{1-\kappa}},
\end{split}
\end{equation*}
where the last series is convergent since $\lambda>{(3n+2\alpha)}/n>3$. The rest of the
proof is exactly the same as that of Theorem 1.1, and we finally obtain
\begin{equation*}
J_2\le w(B)\le\frac{C\cdot w(B)^\kappa}{\lambda}\|f\|_{L^{1,\kappa}(w)}.
\end{equation*}
Combining the above estimates for $J_1$ and $J_2$, and then taking the supremum over all balls $B\subseteq\mathbb R^n$ and all $\sigma>0$, we conclude the proof of Theorem 1.2.
\end{proof}


\begin{thebibliography}{99}

\bibitem{adams} D. R. Adams, A note on Riesz potentials, Duke Math. J, \textbf{42}(1975), 765--778.
\bibitem{chang} S. Y. A. Chang, J. M. Wilson and T. H. Wolff, Some weighted norm inequalities concerning the Schr\"odinger operators, Comment. Math. Helv,  \textbf{60}(1985), 217--246.
\bibitem{chanillo} S. Chanillo and R. L. Wheeden, Some weighted norm inequalities for the area integral, Indiana Univ. Math. J, \textbf{36}(1987), 277--294.
\bibitem{chiarenza} F. Chiarenza and M. Frasca, Morrey spaces and Hardy-Littlewood maximal function, Rend. Math. Appl, \textbf{7}(1987), 273--279.
\bibitem{fan} D. S. Fan, S. Z. Lu and D. C. Yang, Regularity in Morrey spaces of strong solutions to nondivergence
    elliptic equations with VMO coefficients, Georgian Math. J, \textbf{5}(1998), 425--440.
\bibitem{fazio1} G. Di Fazio and M. A. Ragusa, Interior estimates in Morrey spaces for strong solutions to
    nondivergence form equations with discontinuous coefficients, J. Funct. Anal, \textbf{112}(1993), 241--256.
\bibitem{fazio2} G. Di Fazio, D. K. Palagachev and M. A. Ragusa, Global Morrey regularity of strong solutions to the
    Dirichlet problem for elliptic equations with discontinuous coefficients, J. Funct. Anal, \textbf{166}(1999),
    179--196.
\bibitem{garcia2} J. Garcia-Cuerva and J. Rubio de Francia, Weighted Norm Inequalities and Related Topics, North-Holland, Amsterdam, 1985.
\bibitem{gundy} R. F. Gundy and R. L. Wheeden, Weighted integral inequalities for nontangential maximal function,
    Lusin area integral, and Walsh-Paley series, Studia Math, \textbf{49}(1974), 107--124.
\bibitem{huang} J. Z. Huang and Y. Liu, Some characterizations of weighted Hardy spaces, J. Math. Anal. Appl, \textbf{363}(2010), 121--127.
\bibitem{komori} Y. Komori and S. Shirai, Weighted Morrey spaces and a singular integral operator, Math. Nachr,
    \textbf{282}(2009), 219--231.
\bibitem{lerner} A. K. Lerner, Sharp weighted norm inequalities for Littlewood-Paley operators and singular integrals, Adv. Math., \textbf{226}(2011), 3912--3926.
\bibitem{morrey} C. B. Morrey, On the solutions of quasi-linear elliptic partial differential equations, Trans. Amer.
    Math. Soc, \textbf{43}(1938), 126--166.
\bibitem{muckenhoupt1} B. Muckenhoupt, Weighted norm inequalities for the Hardy maximal function, Trans. Amer. Math.
    Soc, \textbf{165}(1972), 207--226.
\bibitem{peetre} J. Peetre, On the theory of $\mathcal L_{p,\lambda}$ spaces, J. Funct. Anal, \textbf{4}(1969),
    71--87.
\bibitem{stein2} E. M. Stein, On some functions of Littlewood-Paley and Zygmund, Bull. Amer. Math.
    Soc, \textbf{67}(1961), 99--101.
\bibitem{stein} E. M. Stein, Singular Integrals and Differentiability Properties of Functions, Princeton Univ. Press, Princeton, New Jersey, 1970.
\bibitem{wang1} H. Wang, The boundedness of some operators with rough kernel on the weighted Morrey spaces, Acta Math. Sinica (Chin. Ser), to appear.
\bibitem{wang5} H. Wang, Intrinsic square functions on the weighted Morrey spaces, J. Math. Anal. Appl, to appear.
\bibitem{wang6} H. Wang, Boundedness of intrinsic square functions on the weighted weak Hardy spaces, preprint, 2012.
\bibitem{wang7} H. Wang, Boundedness of fractional integral operators with rough kernels on weighted Morrey spaces, preprint, 2012.
\bibitem{wang2} H. Wang and H. P. Liu, Some estimates for Bochner-Riesz operators on the weighted Morrey spaces, Acta
    Math. Sinica (Chin. Ser), \textbf{55}(2012), 551--560.
\bibitem{wang3} H. Wang and H. P. Liu, The intrinsic square function characterizations of weighted Hardy spaces, Illinois J. Math, to appear.
\bibitem{wang4} H. Wang and H. P. Liu, Weak type estimates of intrinsic square functions on the weighted Hardy spaces, Arch. Math., \textbf{97}(2011), 49--59.
\bibitem{wilson1} M. Wilson, The intrinsic square function, Rev. Mat. Iberoamericana, \textbf{23}(2007), 771--791.
\bibitem{wilson2} M. Wilson, Weighted Littlewood-Paley Theory and Exponential-Square Integrability, Lecture Notes in Math, Vol 1924, Springer-Verlag, 2007.

\end{thebibliography}
\end{document}